\documentclass[12pt]{article}

\usepackage[centertags]{amsmath}
\usepackage{amsfonts}
\usepackage{amssymb}
\usepackage{amsthm}
\usepackage[misc]{ifsym}
\usepackage{mathrsfs}
\usepackage{graphicx}
\usepackage{epstopdf}
\usepackage{pgf,tikz}
\usetikzlibrary{arrows}
\usepackage{lineno}
\usepackage{authblk}

\theoremstyle{plain}
\newtheorem{thm}{Theorem}[section]
\newtheorem{cor}[thm]{Corollary}
\newtheorem{lem}[thm]{Lemma}

\theoremstyle{definition}
\newtheorem{defn}[thm]{Definition}
\newtheorem{exam}[thm]{Example}
\newtheorem{rem}[thm]{Remark}
\theoremstyle{remark}
\numberwithin{equation}{section}

\newcommand{\beast}{\begin{eqnarray*}}
\newcommand{\eeast}{\end{eqnarray*}}
\begin{document}

\title{A matrix realization of spectral bounds \\ of the spectral radius of a nonnegative matrix}

\author[*]{Yen-Jen Cheng \Letter}
\author[*]{Chih-wen Weng}
\affil[*]{Department of Applied Mathematics, National Chiao Tung University, 1001 University Road, Hsinchu, Taiwan}

\maketitle

\begin{abstract}
We  realize many sharp spectral bounds of the spectral radius of a nonnegative square matrix $C$ by using the largest real eigenvalues of  suitable matrices of smaller sizes related to $C$ that are very easy to find.
As applications, we give a sharp upper bound of the spectral radius of $C$ expressed by the sum of entries, the largest off-diagonal entry $f$ and the largest diagonal entry $d$ in $C$.
We also give a new class of sharp lower bounds of the spectral radius of $C$ expressed
 by the above $d$ and $f$, the least row-sum $r_n$ and the $t$-th largest row-sum $r_t$ in $C$ satisfying
 $0<r_n-(n-t-1)f-d\leq r_t-(n-t)f$, where $n$ is the size of $C$.
\end{abstract}

{\bf keywords}:  nonnegative matrices, spectral radius, spectral bounds

{\bf MSC2010}:  05C50, 15A42

\bigskip

\section{Introduction}\label{s1}

For real matrices $C=(c_{ij})$, $C'=(c'_{ij})$ of the same size, $C'$ {\it majories} $C$, in the notation  $C\leq C'$, if $c_{ij}\leq c'_{ij}$ for all $i$,$j$. When $C$ is a square matrix, the spectral radius $\rho(C)$ of $C$ is defined by
$$\rho(C):=\max\{~|\lambda|~\ |~\lambda\hbox{ is an eigenvalue of $C$}\},$$
where $|\lambda|$ is the magnitude of complex number $\lambda.$
This paper is motivated by the following theorem of   Xing Duan and Bo Zhou in 2013 \cite[Theorem 2.1]{dz}.

\begin{thm}\label{dz2013}
Let $C=(c_{ij})$ be a nonnegative $n\times n$ matrix with  row-sums $r_1\geq r_2\geq \cdots \geq r_n$,  $f:=\max_{1\leq i\not=j\leq n} c_{ij}$ and $d:=\max_{1\leq i\leq n} c_{ii}$. Then
\begin{equation}\label{mo}\rho(C) \leq \frac{r_{\ell}+d-f+\sqrt{(r_{\ell}-d+f)^{2}+4f\sum_{i=1}^{\ell-1}(r_i-r_\ell)}}{2}\end{equation}
for $1 \leq \ell \leq n.$
Moreover, if $C$ is irreducible, then the equality holds in (\ref{mo}) if and only if $r_1=r_n$ or for  $1\leq t\leq \ell$ with $r_{t-1}\not=r_t=r_\ell$, we have $r_t=r_n$ and
$$c_{ij}=\left\{
           \begin{array}{ll}
             d, & \hbox{if $i=j\leq t-1$;} \\
             f, & \hbox{if $i\not=j$ and $1\leq i\leq n$, $1\leq j\leq t-1.$}
           \end{array}
         \right.$$
\end{thm}

Theorem~\ref{dz2013} generalizes the results in \cite{bh:85, cls:13, h:98, hsf:01, lw13, s87, sw:04} and
relates to the  results in \cite{hw:13, Liu, s87}, while the upper bound of $\rho(C)$ expressed in (\ref{mo}) is somewhat complicate and deserves an intuitive realization.

The values on the right hand side of (\ref{mo}) is realized as the largest real eigenvalue $\rho_r(C')$ of the $n\times n$ matrix
\begin{equation}\label{ne1.0}C'=\left(\begin{array}{cccc|ccccc}
d     &f      &\cdots &f      &f      & f        &\cdots &f       &r_1-d-(n-2)f\\
f     &d      &       &f      &f      & f        &\cdots &f       &r_2-d-(n-2)f\\
\vdots&       &\ddots &\vdots &\vdots &\vdots    &       &\vdots  &\vdots      \\
f     &f      &\cdots &d      & f     & f        &\cdots &f       &r_{\ell-1}-d-(n-2)f\\ \hline
f     &f      &\cdots &f      & d     & f        &\cdots &f       &r_{\ell}-d-(n-2)f\\
f     &f      &\cdots &f      & f     & d        &       &f       &r_{\ell}-d-(n-2)f\\
\vdots&\vdots &       &\vdots &\vdots &          &\ddots &\vdots  &\vdots\\
f     &f      &\cdots &f      & f     & f        &       &d       &r_{\ell}-d-(n-2)f\\
f     &f      &\cdots &f      & f     & f        &\cdots &f       &r_{\ell}-(n-1)f\\
\end{array}\right)\end{equation}
which has the following three properties:
\begin{enumerate}
\item[(i)]  The row-sum vector $(r_1, r_2, \ldots, r_\ell, \ldots,  r_\ell)^T$ of $C'$ majories the row-sum vector  $(r_1,$ $r_2,$ $\ldots,$ $r_n)^T$ of $C$,
\item[(ii)] $C'$ majories $C$ except the last column, and
\item[(iii)] $C'$ has a positive eigenvector $(v'_1, v'_2, \ldots, v'_n)^T$ for $\rho_r(C')$ with $v'_i\geq v'_n$ for $1\leq i\leq n.$
\end{enumerate}
Property (iii) will be checked by Lemma~\ref{l4.4}. Since the above matrix $C'$ is not necessary to be nonnegative, the spectra radius $\rho(C')$ of $C'$ is replaced by the largest real eigenvalue $\rho_r(C')$ in the property (iii).
 Our main result in Theorem~\ref{t3.1} is in a more general form that will imply for any matrices $C'$ that satisfy the properties (i)-(iii) above, we have $\rho(C)\leq \rho_r(C').$
Moreover, when the matrix $C'$ and so the value $\rho_r(C')$ are fixed, the matrices $C$ with $\rho(C)=\rho_r(C')$ are completely determined.
We apply Theorem~\ref{t3.1} to find a sharp upper bound of $\rho(C)$ expressed by the sum of entries in $C$, the largest off-diagonal entry $f$ and the largest diagonal entry $d$ in Theorem~\ref{tt6.1}.

 Note that $\rho_r(C')=\rho_r(C'')$ for the largest real eigenvalues of $C'$ and $C''$ respectively, where  $C'$ is as in (\ref{ne1.0}) and
\begin{equation}\label{ne1.2}C''=\left(\begin{array}{cccc|c}
d     &f&\cdots&f     &r_1-d-(\ell-2)f\\
f     &d&      &f     &r_2-d-(\ell-2)f\\
\vdots& &\ddots&\vdots &\vdots \\
f     &f&\cdots&d        &r_{\ell-1}-d-(\ell-2)f\\ \hline
f     &f&\cdots&f     &r_{\ell}-(\ell-1)f
\end{array}\right)
\end{equation}
 is the equitable quotient matrix of $C'$ with respect to the partition $\{\{1\}$, $\{2\}$, $\ldots$, $\{\ell-1\}$, $\{\ell, \ell+1, \ldots, n\}\}$ of $\{1, 2, \ldots, n\}$.
 Moreover  $\rho_r(C'')=\rho_r(C''')$, where
\begin{equation}\label{n31.3}C'''=\begin{pmatrix}(\ell-2)f+d & f\\
\sum_{i=1}^{\ell-1} r_i-(\ell-1)((\ell-2)f+d)    & r_\ell-(\ell-1)f
\end{pmatrix},\end{equation}
is the equitable quotient of the transpose $C''^T$ of $C''$ with respect to the partition $\{\{1, 2, \ldots, \ell-1\}, \{\ell\}\}$ of $\{1, 2, \ldots, \ell\}$.
Motivated by these observations, Theorem~\ref{main3} will provide an upper bound $\rho_r(C'')$ of $\rho(C)$, where $C''$ is a matrix of size smaller than that of $C$ obtained by applying equitable quotient to suitable matrix $C'$  that satisfies properties (i)-(iii) described above.

Every our theorem of upper bounds of $\rho(C)$ has a dual version that deals with lower bounds. We provide a new class of sharp lower bounds of $\rho(C)$  in Theorem~\ref{main5}. Applying  Theorem~\ref{main5} to a binary matrix $C$, we improve the well known inequality  $\rho(C)\geq r_n$ as stated in Corollary~\ref{cor10.2}.
We believe that many new spectral bounds of the spectral radius of a nonnegative matrix will be easily obtained by our matrix realization in this paper.

In addition to the above results, Lemma~\ref{nt3.1}, Lemma~\ref{nt3.1'} and Lemma~\ref{lem2.2'} are of independent interest in matrix theory.

\section{Preliminaries}\label{s1.5}

Our study is based on the famous Perron-Frobenius Theorem, hence we shall review the necessary parts of the theorem in this section.

\begin{thm}[{\cite[Theorem 2.2.1]{Brou}, \cite[Corollary 8.1.29, Theorem 8.3.2]{Horn}}]\label{PF}
If $C$ is a nonnegative square matrix, then the following (i)-(iii) hold.
\begin{enumerate}
\item[(i)] The spectral radius $\rho(C)$ is an eigenvalue of $C$ with a corresponding nonnegative right eigenvector and a corresponding nonnegative left eigenvector.
\item[(ii)] If there exists a column vector $v> 0$ and a nonnegative number $\lambda$ such that $Cv\leq \lambda v$, then $\rho(C)\leq\lambda$.
\item[(iii)] If there exists a column vector $v\geq 0$, $v\not=0$  and a nonnegative number $\lambda$ such that  $Cv\geq \lambda v$, then  $\rho(C)\geq\lambda$.
\end{enumerate}
Moreover, if in addition  $C$ is irreducible, then
 the eigenvalue $\rho(C)$ in (i) has multiplicity $1$ and its corresponding left eigenvector and right eigenvector can be chosen to be positive, and any nonnegative left or right eigenvector of $C$ is only corresponding to the eigenvalue $\rho(C)$.
 \qed
\end{thm}

Without further mention, an eigenvector is always a right eigenvector.
The following two lemmas are  well-known consequences of Theorem~\ref{PF}. We shall provide their proofs since they motivate our  proofs of results.

\begin{lem}[{\cite[Theorem 2.2.1]{Brou}}]\label{lem1.2}
If $0\leq C\leq C'$ are square matrices, then $\rho(C)\leq\rho(C')$. Moreover, if $C'$ is irreducible,
then $\rho(C')=\rho(C)$  if and only if $C'=C$.
\end{lem}
\begin{proof}
 Let $v$ be  a nonnegative eigenvector of $C$ for $\rho(C)$. From the assumption,  $C'v\geq Cv=\rho(C)v$.
 By Theorem~\ref{PF}(iii) with $(C, \lambda)=(C', \rho(C))$, we have  $\rho(C')\geq \rho(C).$
   Clearly $C'=C$ implies $\rho(C')=\rho(C)$.
 If $\rho(C')=\rho(C)$ and $C'$ is irreducible, then
 $\rho(C)v'^Tv=\rho(C') v'^Tv=v'^TC'v\geq v'^TCv= \rho(C)v'^Tv,$
 where $v'^T$ is a positive left eigenvector of $C'$ for $\rho(C').$
 Hence the above inequality is the equality $v'^TC'v=v'^TCv$.  As $v'^T$ positive,
  $C'v=Cv=\rho(C)v=\rho(C')v,$ so $v$ is a positive eigenvector of $C'$ for $\rho(C')$ and $C'=C.$
 \end{proof}

The matrix $C'$ in Lemma~\ref{lem1.2} is a {\it matrix realization} of the upper bound $\rho(C')$ of $\rho(C)$. We shall provide other matrix realizations as stated in the title.

\begin{lem}[{\cite[Theorem 8.1.22]{Horn}}]\label{lem1.1}
If an $n\times n$ matrix $C=(c_{ij})$ is nonnegative with row-sum vector $(r_1,r_2,\ldots,r_n)^T$, where $r_i=\sum_{1\leq j\leq n}c_{ij}$ and $r_1\geq r_i\geq r_n$ for $1\leq i\leq n$,
then
$$r_n \leq \rho(C) \leq r_1.$$
Moreover, if $C$ is irreducible, then $\rho(C)=r_1$ (resp.  $\rho(C)=r_n$) if and only if $C$ has constant row-sum.
\end{lem}

We provide a proof of the following generalized version of Lemma~\ref{lem1.1}, which is due to  M. N. Ellingham and Xiaoya Zha \cite{ez}.

\begin{lem}[\cite{ez}]\label{lem1.1'}
If an $n\times n$ matrix $C$ with row-sum vector $(r_1,r_2,\ldots,r_n)^T$, where $r_1\geq r_i\geq r_n$ for $1\leq i\leq n$, has a nonnegative left eigenvector $v^T=(v_1,v_2,\ldots,v_n)$ for $\theta$, then
$$r_n \leq \theta \leq r_1.$$
Moreover, $\theta=r_1$ (resp. $\theta=r_n$) if and only if $r_i=r_1$ (resp. $r_i=r_n$) for the indices $i$ with $v_i\ne 0$.
In particular, if $v^T$ is positive, $\theta=r_1$ (resp. $\theta=r_n$) if and only if $C$ has constant row-sum.
\end{lem}
\begin{proof}
Without loss of generality, let $\sum_{i=1}^{n}v_i=1$ and $u$ be the all-one column vector. Then
$$\theta=\theta v^Tu=v^TCu=\sum_{i=1}^nv_ir_i.$$
So $\theta$ is a convex combination of those $r_i$ with indices $i$ satisfying $1\leq i\leq n$ and $v_i> 0$, and the lemma follows.
\end{proof}

In the sequels, we shall call two statements  that resemble each other by switching $\leq$ and $\geq$ and corresponding variables, like
$\theta\geq r_n$ and $\theta\leq r_1$, as {\it dual statements}, and  their proofs are called {\it dual proofs} if one proof is obtained from the other  by simply switching one of $\leq$ and $\geq$ to the other.

\section{A generalization of Lemma~~\ref{lem1.2}}\label{s2.2}

We generalize Lemma~\ref{lem1.2} in the sense of Lemma~\ref{lem1.1'} that the matrices considered are not necessary to be nonnegative.
\begin{lem}\label{nt3.1}
 Let $C=(c_{ij})$, $C'=(c'_{ij})$, $P$ and $Q$ be  $n\times n$ matrices.
Assume that
\begin{enumerate}
\item[(i)]    $PCQ\leq PC'Q$;
\item[(ii)]  $C'$ has an eigenvector $Qu$ for $\lambda'$ for some nonnegative column vector $u=(u_1, u_2, \ldots, u_n)^T$  and $\lambda'\in \mathbb{R}$;
\item[(iii)] $C$ has a left eigenvector $v^TP$ for $\lambda$ for some nonnegative row vector $v^T=(v_1, v_2, \ldots, v_n)$  and  $\lambda\in \mathbb{R}$; and
\item[(iv)] $v^TPQu>0.$
\end{enumerate}
 Then $\lambda\leq \lambda'$.
Moreover, $\lambda=\lambda'$
if and only if
\begin{equation}\label{ne4.1}
(PC'Q)_{ij}=(PCQ)_{ij}\qquad \hbox{for~}1\leq i, j\leq n \hbox{~with~} v_i\ne 0 \hbox{~and~} u_j\ne 0.
\end{equation}
\end{lem}

\begin{proof}
Multiplying the nonnegative vector $u$ in (ii) to the right of both terms of  (i),
\begin{equation}\label{e4.3}
PCQu\leq PC'Qu=\lambda'PQu.
\end{equation}
Multiplying the nonnegative left eigenvector $v^T$ of $C$ for $\lambda$ in assumption (iii) to the left of all terms  in (\ref{e4.3}), we have
\begin{equation}\label{e4.4}\lambda v^TPQu=v^TPCQu\leq v^TPC'Qu=\lambda' v^TPQu.\end{equation}
Now delete the positive term $v^TPQu$ by assumption (iv) to obtain $\lambda\leq \lambda'$ and finish the proof of the first part.

 Assume that $\lambda=\lambda'$, so the inequality in (\ref{e4.4}) is an equality.  Especially $(PCQu)_i=(PC'Qu)_i$
for any $i$ with $v_i\not=0.$ Hence $(PCQ)_{ij}=(PC'Q)_{ij}$ for  any $i$ with $v_i\not=0$ and any $j$ with $u_j\not=0.$

Conversely, (\ref{ne4.1}) implies $$v^TPCQu=\sum_{i,j} v_i(PCQ)_{ij}u_j=\sum_{i,j} v_i(PC'Q)_{ij}u_j=v^TPC'Qu,$$ so
$\lambda=\lambda'$ by (\ref{e4.4}).
\end{proof}

If $C$ is nonnegative and $P=Q=I$, where $I$ is  the $n\times n$ identity matrix, then Lemma~\ref{nt3.1} becomes Lemma~\ref{lem1.2}
with an additional assumption $v^Tu>0$ which immediately holds if $C$ or $C'$ is irreducible by Theorem~\ref{PF}.
The following is a dual version of lemma~\ref{nt3.1} and its proof is by dual proof.

\begin{lem}\label{nt3.1'}
 Let $C=(c_{ij})$, $C'=(c'_{ij})$, $P$ and $Q$ be  $n\times n$ matrices.
Assume that
\begin{enumerate}
\item[(i)]    $PCQ\geq PC'Q$;
\item[(ii)]  $C'$ has an eigenvector $Qu$ for $\lambda'$ for some nonnegative column vector $u=(u_1, u_2, \ldots, u_n)^T$  and $\lambda'\in \mathbb{R}$;
\item[(iii)] $C$ has a left eigenvector $v^TP$ for $\lambda$ for some nonnegative row vector $v^T=(v_1, v_2, \ldots, v_n)$  and  $\lambda\in \mathbb{R}$; and
\item[(iv)] $v^TPQu>0.$
\end{enumerate}
 Then $\lambda\geq \lambda'$.
Moreover, $\lambda=\lambda'$
if and only if
\begin{equation}\label{ne4.1'}
(PC'Q)_{ij}=(PCQ)_{ij}\qquad \hbox{for~}1\leq i, j\leq n~ \hbox{~with~} v_i\ne 0 \hbox{~and~} u_j\ne 0.
\end{equation}\qed
\end{lem}

\section{The special case $P=I$ and a particular $Q$}\label{s2.5}

We shall applying Lemma~\ref{nt3.1} and Lemma~\ref{nt3.1'} by using $P=I$ and
\begin{equation}\label{defnq}
Q=\begin{pmatrix}
1 & 0 & \cdots & 0 & 1 \\
0 & 1 & \cdots & 0 & 1 \\
\vdots & \vdots & \ddots & \vdots & \vdots \\
0 & 0 & \cdots & 1 & 1 \\
0 & 0 & \cdots & 0 & 1 \\
\end{pmatrix}.
\end{equation}
Hence for $n\times n$ matrix $C'=(c'_{ij}),$ the matrix $PC'Q$ in Lemma~\ref{nt3.1}(i) is
\begin{equation}\label{C'Q}C'Q=\begin{pmatrix}
c'_{11} & c'_{12} & \cdots & c'_{1\ n-1} & r'_1\\
c'_{21} & c'_{22} & \cdots & c'_{2\ n-1} & r'_2\\
\vdots & \vdots & \ddots & \vdots & \vdots \\
c'_{n-1\ 1} & c'_{n-1\ 2} & \cdots & c'_{n-1\ n-1} & r'_{n-1}\\
c'_{n1} & c'_{n2} & \cdots & c'_{n\ n-1} & r'_n\\
\end{pmatrix},
\end{equation}
where $(r'_1, r'_2, \ldots, r'_n)^T$ is the row-sum column vector of $C'.$

\begin{defn}\label{d3.1}
 A column vector $v'=(v'_1,v'_2,\ldots,v'_n)^T$ is called {\it rooted}  if $v'_j\geq v'_n\geq 0$  for $1\leq j\leq n-1$.
\end{defn}

The following Lemma is immediate from the above definition.
\begin{lem}\label{l3.15}
If $u=(u_1, u_2, \ldots, u_n)^T$ and $v'=(v'_1, v'_2, \ldots, v'_n):=Qu=(u_1+u_n, u_2+u_n, \ldots, u_{n-1}+u_n, u_n)^T$, then
\begin{enumerate}
\item[(i)] $v'$ is rooted  if and only if  $u$ is nonnegative;
\item[(ii)] $u_j>0$ if and only if $v'_j>v'_n$ for $1\leq j\leq n-1$.
\end{enumerate}\qed
\end{lem}

The following matrix notation will be adopted in the paper. For a matrix $C=(c_{ij})$ and subsets $\alpha$, $\beta$ of row indices and column indices respectively, we use $C[\alpha|\beta]$ to denote the submatrix
of $C$ with size $|\alpha|\times |\beta|$ that has entries $c_{ij}$ for $i\in \alpha$ and $j\in\beta$,
$C[\alpha|\beta):=C[\alpha|\overline{\beta}],$ where $\overline{\beta}$ is the complement of $\beta$ in the set of column indices, and
similarly, for the definitions of $C(\alpha|\beta]$ and $C(\alpha|\beta).$
 For $\ell\in \mathbb{N},$ $[\ell]:=\{1, 2, \ldots, \ell\},$ symbol $-$ is the complete set of indices, and we use  $i$ to denote the singleton subset $\{i\}$ to reduce the double use of parentheses. For example of the $n\times n$ matrix $C$,
 $C[-|n)=C[[n]|[n-1]]$ is the $n\times (n-1)$ submatrix of $C$ obtained by deleting the last column of $C$.
The following theorem is immediate from Lemma~\ref{nt3.1} by applying $P=I$, the $Q$ in (\ref{defnq}), $v'=Qu$ and referring to (\ref{C'Q}) and Lemma~\ref{l3.15}.

\begin{thm}\label{t3.1}
 Let $C=(c_{ij})$, $C'=(c_{ij})$ be  $n\times n$ matrices.
Assume that
\begin{enumerate}
\item[(i)]   $C[-|n)\leq C'[-|n)$ and the row-sum vector $(r'_1, r'_2, \ldots, r'_n)^T$ of $C'$ majories the row-sum vector  $(r_1, r_2, \ldots, r_n)^T$ of $C$;
\item[(ii)]  $C'$ has a rooted eigenvector $v'=(v'_1, v'_2, \ldots, v'_n)^T$ for $\lambda'$ for some  $\lambda'\in \mathbb{R}$;
\item[(iii)] $C$ has a nonnegative left eigenvector $v^T=(v_1, v_2, \ldots, v_n)$ for $\lambda\in \mathbb{R}$;
\item[(iv)] $v^Tv'>0.$
\end{enumerate}
 Then $\lambda\leq \lambda'$.
Moreover, $\lambda=\lambda'$
if and only if
\begin{enumerate}
\item[(a)] $r_i=r'_i$\qquad for $1\leq i\leq n$ with $v_i\not=0$ when $v'_n\not=0;$
\item[(b)]
$c'_{ij}=c_{ij}\qquad \hbox{for~}1\leq i\leq n,~1\leq j\leq n-1 \hbox{~with~} v_i\ne 0 \hbox{~and~} v'_j> v'_n.$
\end{enumerate} \qed
\end{thm}

Note that the cases (a)-(b) in Theorem~\ref{t3.1} is from the line (\ref{ne4.1}) in Theorem~\ref{nt3.1}.
The first part of assumption (i) in Theorem~\ref{t3.1}  says that last column is {\it irrelevant} in the comparison of $C$ and $C'$. The following theorem is a dual version of Theorem~\ref{t3.1}.

\begin{thm}\label{t3.1'}
 Let $C=(c_{ij})$, $C'=(c_{ij})$ be  $n\times n$ matrices.
Assume that
\begin{enumerate}
\item[(i)]   $C[-|n)\geq C'[-|n)$ and the row-sum vector $(r_1, r_2, \ldots, r_n)^T$  of $C$ majories the row-sum vector $(r'_1, r'_2, \ldots, r'_n)^T$ of $C'$;
\item[(ii)]  $C'$ has a rooted eigenvector $v'=(v'_1, v'_2, \ldots, v'_n)^T$ for $\lambda'$ for some  $\lambda'\in \mathbb{R}$;
\item[(iii)] $C$ has a nonnegative left eigenvector $v^T=(v_1, v_2, \ldots, v_n)$ for $\lambda\in \mathbb{R}$;
\item[(iv)] $v^Tv'>0.$
\end{enumerate}
 Then $\lambda\geq \lambda'$.
Moreover, $\lambda=\lambda'$
if and only if
\begin{enumerate}
\item[(a)] $r_i=r'_i$\qquad for $1\leq i\leq n$ with $v_i\not=0$ when $v'_n\not=0;$
\item[(b)]
$c'_{ij}=c_{ij}\qquad \hbox{for~}1\leq i\leq n,~1\leq j\leq n-1 \hbox{~with~} v_i\ne 0 \hbox{~and~} v'_j> v'_n.$
\end{enumerate}  \qed
\end{thm}

\begin{exam}\label{exam2}
Consider the following three matrices
$$C'_\ell= \begin{pmatrix}
3 & 1 & 1\\ 0 & 0 & 3 \\ 0 & 1 & 2
\end{pmatrix}, \quad  C=\begin{pmatrix}
3 & 1 & 1\\ 1 & 0 & 2 \\ 1 & 1 & 1
\end{pmatrix}, \quad C'_u=\begin{pmatrix}
3 & 2 & 0\\ 1 & 2 & 0 \\ 1 & 2 & 0
\end{pmatrix}$$
with $C'_\ell[-|3)\leq C[-|3) \leq C'_u[-|3),$ and the same row-sum vector $(5, 3, 3)^T$.
Note that $C'_\ell$ has a rooted eigenvector $v'^\ell=(1, 0, 0)^T$ for $\lambda'^\ell=3$
and $C'_u$ has a rooted eigenvector $v'^u=(2, 1, 1)^T$ for $\lambda'^r=4$.
Since $C$ is irreducible, it has a left positive eigenvector $(v_1,v_2,v_3)>0$.
Hence assumptions (i)-(iv) in Theorem~\ref{t3.1} and Theorem~\ref{t3.1'} hold, and
we conclude that $\lambda'^\ell\leq \rho(C)\leq \lambda'^r$.
Since $[3]\times [1]$ is the set of the pairs $(i,j)$ described in Theorem~\ref{t3.1}(b) and Theorem~\ref{t3.1'}(b),
from simple comparison of
the first columns $C'_\ell[-|~1]< C[-|~1] = C'_u[-|~1]$ of these three matrices,
we easily conclude that $3=\lambda'^\ell< \rho(C)= \lambda'^r=4$ by the second part of Theorem~\ref{t3.1} and that of Theorem~\ref{t3.1'}.
\end{exam}

\section{Matrices with a rooted eigenvector}\label{ns4}

Before giving applications of Theorem~\ref{t3.1} and Theorem \ref{t3.1'}, we need to construct $C'$ which possesses a rooted eigenvector for some $\lambda'$. The following lemma comes immediately.

\begin{lem}\label{l4.1}
If a square matrix $C'$ has a rooted eigenvector for $\lambda'$, then $C'+dI$ also has
the same rooted eigenvector for $\lambda'+d,$ where $d$ is a constant and $I$ is the identity matrix with the same size of $C'$. \qed
\end{lem}

 A rooted column vector defined in Definition~\ref{d3.1} is generalized to a rooted matrix as follows.
\begin{defn}\label{def4.1}
A  matrix $C'=(c'_{ij})$ is called {\it rooted}  if its  columns and its row-sum vector are all rooted except the last column of $C'$.
\end{defn}

The vertex $Q$ in (\ref{defnq}) is invertible with
$$Q^{-1}=\begin{pmatrix}
1 & 0 & \cdots & 0 & -1 \\
0 & 1 & \cdots & 0 & -1 \\
\vdots & \vdots & \ddots & \vdots & \vdots \\
0 & 0 & \cdots & 1 & -1 \\
0 & 0 & \cdots & 0 & 1 \\
\end{pmatrix}.$$ Multiplying $Q^{-1}$ to $C'Q$ in (\ref{C'Q}),   $Q^{-1}C'Q$ is
\begin{equation}\label{e3.1}\begin{pmatrix}
c'_{11}-c'_{n1}     & c'_{12}-c'_{n2} & \cdots     & c'_{1\ n-1}-c'_{n\ n-1} & r'_1-r'_n \\
c'_{21}-c'_{n1}     & c'_{22}-c'_{n2} & \cdots     & c'_{2\ n-1}-c'_{n\ n-1} & r'_2-r'_n \\
\vdots              & \vdots & \ddots              & \vdots & \vdots \\
c'_{n-1\ 1}-c'_{n1} & c'_{n-1\ 2}-c'_{n2} & \cdots & c'_{n-1\ n-1}-c'_{n\ n-1} & r'_{n-1}-r'_{n} \\
c'_{n1}             & c'_{n2} & \cdots             & c'_{n\ n-1} & r'_n \\
\end{pmatrix}.
\end{equation}

The matrices $C'$ and $Q^{-1}C'Q$ have the same set of eigenvalues. Moreover, $v'$ is an eigenvector of $C'$ for $\lambda'$ if and only if $u=Q^{-1}v'$ is an eigenvector of $Q^{-1}C'Q$ for $\lambda'$. From (\ref{e3.1}),
  $C'$ is rooted if and only if $Q^{-1}C'Q$ is nonnegative.
The first part of the following lemma follows immediately from the above discussion and Theorem~\ref{PF} by choosing $\lambda'=\rho(C')$.

\begin{lem}\label{l5.3}
If $C'$ is a rooted matrix, then $Q^{-1}C'Q$ is nonnegative, $\rho(C')$ is an eigenvalue of $C'$,
and $C'$ has a rooted eigenvector $v'=Qu$ for $\rho(C')$,
where $u$ is a nonnegative eigenvector of $Q^{-1}C'Q$ for $\rho(C')$.
Moreover, with $v'=(v'_1, v'_2, \ldots, v'_n)^T$,  the following (i)-(ii) hold.
\begin{enumerate}
\item[(i)] If $C' [n|n)$ is positive, then $v'$ is positive.
\item[(ii)] If  $C' [n|n)$ is positive and  $r'_i> r'_n$ for all $1\leq i\leq n-1$, then $v'_j>v'_n$ for all $1\leq j\leq n-1.$
\end{enumerate}
\end{lem}
\begin{proof}
It remains to prove the second part.

(i) suppose that $C' [n|n)$ is positive and $v'_n=0$. Then
$$\sum_{j=1}^{n-1}c'_{nj}v'_j=\sum_{j=1}^nc'_{nj}v'_j=(C'v')_n=\rho(C')v'_n=0.$$
Hence $v'$ is a zero vector since $c'_{nj}>0$ for $j\leq n-1$, a contradiction. So $v'_n>0$ and $v'>0$ since $v'$ is rooted.

(ii) The assumptions imply that the matrix $Q^{-1}C'Q$ in (\ref{e3.1}) is irreducible.
Hence $u$ is positive. By Lemma~\ref{l3.15}(ii), $v'_j>v'_n$ for $1\leq j<n.$
\end{proof}

The largest real eigenvalue of the following matrix will be used to obtain  bounds of the spectral radius of a nonnegative matrix.

Fix  $d,f,r_1,r_2,\ldots,r_n\geq 0$ such that $r_j\geq r_n$ for $1\leq j\leq n-1$, and let
\begin{equation}\label{ee5.2}
M_n(d, f, r_1, r_2, \ldots, r_n)=
\begin{pmatrix}
d & f &  \cdots & f & r_{1}-(d+(n-2)f) \\
f & d &   & f & r_{2}-(d+(n-2)f) \\
\vdots &  & \ddots & \vdots & \vdots \\
f & f & \cdots  & d & r_{n-1}-(d+(n-2)f)\\
f & f & \cdots & f & r_{n}-(n-1)f
\end{pmatrix}
\end{equation} be an $n\times n$ matrix with row-sum vector $(r_1, r_2, \ldots, r_n)^T$.

Note that for any square matrix $C'$, it might be $\rho(C'+dI)\not=\rho(C')+d$, but $\rho_r(C'+dI)=\rho_r(C')+d$ always holds, where $\rho_r(C'+dI)$ and $\rho_r(C')$ are the largest real eigenvalues of $C'+dI$ and $C'$ respectively.
Also $\rho(C')=\rho_r(C')$ if $C'$ is nonnegative.

\begin{lem}\label{l4.4} The following (i)-(ii) hold.
\begin{enumerate}
\item[(i)] The matrix $M_n(d, f, r_1, r_2, \ldots, r_n)$ has a rooted eigenvector $v'=(v'_1, v'_2, \ldots, v'_n)^T$ for the largest real eigenvalue $\rho_r(M_n(d, f, r_1, r_2, \ldots, r_n))$ of $M_n(d, f, r_1, r_2, \ldots, r_n)$.
\item[(ii)] If $f>0$, then $v'>0$.
\end{enumerate}
\end{lem}
\begin{proof} Let $M_n:=M_n(d, f, r_1, r_2, \ldots, r_n)$.
First assume $d\geq f$. Then $M_n$ is rooted. (i)-(ii) follows from (i)-(ii) of Lemma~\ref{l5.3},  in particular $\rho(M_n)=\rho_r(M_n)$.
If  $d<f$, then  the matrix $(f-d)I+M_n$  is a rooted matrix.
As in the first part, let $v'$ be a rooted eigenvector of $(f-d)I+M_n$ for $\rho((f-d)I+M_n).$ Note that $v'$ is also a rooted eigenvector of $M_n$ for $\rho_r(M_n)=\rho((f-d)I+M_n)-(f-d).$ This proves (i), and (ii) follows similarly  from (ii) of Lemma~\ref{l5.3}.
\end{proof}

\section{Equitable partition}\label{s3}

Pattern in $C'$ will make it easier to compute its eigenvalue and to find the bound $\lambda'$ obtained in Theorem~\ref{t3.1} and Theorem~\ref{t3.1'}.
For a partition $\Pi=\{\pi_1, \pi_2, \ldots, \pi_\ell\}$ of $[n]$,
 the $\ell\times \ell$ matrix $\Pi(C'):=(\pi'_{ab})$, where
$$\pi'_{ab}:= \frac{1}{|\pi_a|}\sum_{i\in \pi_a, j\in \pi_b} c'_{ij},$$
is called the {\it quotient matrix} of $C'$ with respect to $\Pi$. In matrix notation,
\begin{equation}\label{e5.1}
\Pi(C')=(S^TS)^{-1}S^TC'S,
\end{equation}
where $S=(s_{jb})$ is the $n\times \ell$ {\it characteristic matrix} of $\Pi$, i.e.,
$$s_{jb}=
\left\{
\begin{array}{ll}
1, & \hbox{if $j\in \pi_b$;} \\
0, & \hbox{otherwise.} \\
\end{array}
\right.$$
for $1\leq j\leq n,$ and $1\leq b\leq \ell.$
If
$$\pi'_{ab}=\sum_{j\in \pi_b} c'_{ij}\qquad (1\leq a, b\leq \ell)$$
for all $i\in \pi_a,$ then  $\Pi(C')=(\pi'_{ab})$ is called
the {\it equitable quotient matrix} of $C'$ with respect to $\Pi.$
Note that $\Pi(C')$ is an equitable quotient matrix if and only if  \begin{equation}\label{e5.2}S\Pi(C')=C'S.\end{equation}

\begin{lem}[{\cite[Lemma 2.3.1]{Brou}}]\label{q1}
If an $n\times n$ matrix $C'$ has an equitable quotient matrix $\Pi(C')$ with respect to partition $\Pi=\{\pi_1, \pi_2, \ldots, \pi_\ell\}$ of $[n]$ with characteristic matrix $S$,  and $\lambda'$ is an eigenvalue of  $\Pi(C')$ with eigenvector $u'$,   then   $\lambda'$ is an  eigenvalue of $C'$ with eigenvector  $Su'$. Moreover, if $u'$ is rooted and $n\in \pi_\ell$, then $Su'$ is rooted.
\end{lem}
\begin{proof} From (\ref{e5.2}),
$C'Su'=S\Pi(C')u'=\lambda' Su'.$ The second statement is clear.
\end{proof}

Under some conditions, the spectral radius is preserved by equitable quotient operation.

\begin{lem}\label{lem2.2'} If $\Pi=\{\pi_1,\ldots, \pi_\ell\}$ is a partition of $[n]$ and
  $C'=(c'_{ij})$ is an $n\times n$ matrix satisfying
 $c'_{ij}=c'_{kj}$ for all $i, k$ in the same part $\pi_a$ of $\Pi$ and $j\in [n],$
 then $C'$ and its quotient matrix $\Pi(C')$ with respect to $\Pi$ have the same set of nonzero eigenvalues.
 In particular, $\rho(C')=\rho(\Pi(C'))$.
\end{lem}
\begin{proof} From the construction of $C'$, $\Pi(C')$ is clear to be an equitable quotient matrix of $C'$.
Let $\lambda'$ be a nonzero eigenvalue of $C'$ with eigenvector $v'=(v'_1,\ldots,v'_n)^T$.
Then $v'_i=(Cv')_i/\lambda'=(Cv')_{k}/\lambda'=v'_{k}$ for all $i, k$  in the same part $\pi_a$ of $\Pi.$
Let $u'_a=v'_i$ with any choice of  $i\in \pi_a.$
Then $u':=(u'_1,\ldots,u'_\ell)\ne 0$, and  $\Pi(C')u'=\lambda' u'$.
 From this and Lemma~\ref{q1}, we know that $C'$ and $\Pi(C')$ have the same set of nonzero eigenvalues, and thus $\rho(C')=\rho(\Pi(C')).$
\end{proof}


\begin{lem}\label{l7.1} For $M_n=M_n(d, f, r_1, r_2, \ldots, r_n)$ defined in \eqref{ee5.2}, where  $d, f\geq 0$ and $r_1\geq r_2\geq \cdots\geq r_n\geq 0$, we have  the following (i)-(iii).
\begin{enumerate}
\item [(i)] The largest real eigenvalue $\rho_r(M_n)$ of $M_n$ satisfies \begin{align*}
    \rho_r(M_n):=&\frac{r_n+d-f+\sqrt{(r_n-d+f)^2+4f\sum_{i=1}^{n-1}(r_i-r_n)}}{2}\\
    \geq& \max(d-f, r_n).\end{align*}
\item[(ii)] If $r_n=0$, then $$\rho_r(M_n)=\frac{d-f+\sqrt{(d-f)^2+4f m}}{2},$$
where $m:=\sum_{i=1}^{n-1} r_i$ is the sum of all entries of $M_n$.
\item [(iii)]  If $r_t=r_n$ for $t\leq n$, then $\rho_r(M_t)=\rho_r(M_n).$
\end{enumerate}
\end{lem}
\begin{proof}
(i)  We consider the matrix $M_n+(f-d)I$. Note that $(M_n+(f-d)I)^T$ has equitable quotient matrix
$$\Pi((M_n+(f-d)I)^T)=\begin{pmatrix} (n-1)f & f \\  \sum_{i=1}^{n-1} (r_i-(d+(n-2)f))  & r_n-(d+(n-2)f) \end{pmatrix}$$
with respect to the  partition $\Pi=\{\{1, 2, \ldots, n-1\}, \{n\} \}$ of $[n]$,
which has two eigenvalues $$\frac{r_n-d+f\pm\sqrt{(r_n-d+f)^2+4f\sum_{i=1}^{n-1}(r_i-r_n)}}{2}.$$
Since $((M_n+(f-d)I)^T)_{ij}=((M_n+(f-d)I)^T)_{kj}$ all $i,k\in [n-1]$ and $j\in [n]$ and by  Lemma~\ref{lem2.2'}, $(M_n+(f-d)I)^T$ has eigenvalues
$$0^{n-2},\frac{r_n-d+f\pm\sqrt{(r_n-d+f)^2+4f\sum_{i=1}^{n-1}(r_i-r_n)}}{2},$$
and $M_n$ has eigenvalues
$$(d-f)^{n-2},\frac{r_n+d-f\pm\sqrt{(r_n-d+f)^2+4f\sum_{i=1}^{n-1}(r_i-r_n)}}{2}.$$
Note that
$$\frac{r_n+d-f+\sqrt{(r_n-d+f)^2+4f\sum_{i=1}^{n-1}(r_i-r_n)}}{2}\geq \max(d-f, r_n).$$
So the proof of (i) is complete.

(ii) and (iii)  follow from (i) immediately.
\end{proof}

\section{Spectral upper bounds with prescribed sum of entries}

Let $J_k$, $I_k$ and $O_{k}$ be the $k\times k$ all-one matrix, the $k\times k$ identity matrix and the $k\times k$ zero matrix respectively. We recall an old  result of Richard Stanley \cite{s87}.

\begin{thm}[\cite{s87}]\label{t6.1s}
Let $C=(c_{ij})$ be an $n\times n$ symmetric (0,1) matrix with zero trace. Let the number of 1's of $C$ be $2e$. Then
$$\rho(C)\leq \frac{-1+\sqrt{1+8e}}{2}.$$
Equality holds if and only if
$$e=\binom{k}{2}$$
and $PCP^T$ has the form
$$
\left(
\begin{matrix}
J_k-I_k & 0 \\
0   & O_{n-k}
\end{matrix}
\right)=(J_k-I_k)\oplus O_{n-k}
$$
for some permutation matrix $P$ and positive integer $k$.
\end{thm}

The following theorem generalizes Theorem~\ref{t6.1s} to nonnegative matrices, not necessary symmetric.
\begin{thm}\label{tt6.1}
Let $C=(c_{ij})$ be an $n\times n$ nonnegative matrix. Let $m$ be the sum of entries
and $d$ (resp. $f$) be any number which is larger than or equal to the largest diagonal element (resp. the largest off-diagonal element) of $M$.
Then
 \begin{equation}\label{e7.1}
\rho(C)\leq \frac{d-f+\sqrt{(d-f)^2+4mf}}{2}.
\end{equation}
Moreover, if $mf>0$ then the equality in (\ref{e7.1}) holds
 if and only if
 $m=k(k-1)f+kd$ and
 $PCP^T$ has the form
 $$\begin{pmatrix}
 fJ_k+(d-f)I_k & 0\\
 0   & O_{n-k}
 \end{pmatrix}=
 (fJ_k+(d-f)I_k)\oplus O_{n-k}$$   for some permutation matrix $P$ and some  nonnegative integer $k$.
\end{thm}

\begin{proof} If $f=0$ then the nonzero entries only appear in the diagonal of $C$, so $\rho(C)\leq d$ and
(\ref{e7.1}) holds. Assume $f>0$ for the remaining.
Consider the $(n+1)\times (n+1)$ nonnegative matrix $M=C\oplus O_1$ which  has row-sum vector $(r_1,r_2,\ldots,r_n,r_{n+1})^T$ with $r_{n+1}=0$ and  a nonnegative left eigenvector $v^T$ for $\rho(M)=\rho(C)$.
Let $C'=M_{n+1}(d, f, r_1, r_2, \ldots, r_{n+1})$ as defined in (\ref{ee5.2}) which has the same row-sum vector of $M$, and has a positive rooted eigenvector $v'=(v'_1, v'_2, \ldots, v'_{n+1})^T$ for $\rho_r(C')$ by Lemma~\ref{l4.4}(i).
Clearly $M[-|n+1)\leq C'[-|n+1)$ and $v^Tv'>0.$ Hence the assumptions (i)-(iv) in Theorem~\ref{t3.1}
hold with $(C, \lambda, \lambda')=(M, \rho(M), \rho_r(C'))$.
 Now by Theorem~\ref{t3.1} and Lemma~\ref{l7.1}(ii), we have
$$\rho(C)=\rho(M)\leq \rho_r(C')=\frac{d-f+\sqrt{(d-f)^2+4mf}}{2},$$
finishing the proof of the first part.

To prove the second part, assume $m=k(k-1)f+kd$ and $PCP^T=(fJ_k+(d-f)I_k)\oplus O_{n-k}$ for one direction.
Using $\rho(C)=\rho(PCP^T)=\rho(fJ_k+(d-f)I_k),$ we have
$$\rho(C)=(k-1)f+d=\frac{d-f+\sqrt{(d-f)^2+4mf}}{2}.$$

For the other direction, assume $\rho(C)=\rho_r(C')$ and $mf>0$. In particular $C\not=0$ and $M\not=0$.
Let $(v_1, v_2, \ldots, v_{n+1})$ be a nonnegative left eigenvector of $M$. Then $v_{n+1}=0$.
 Write $\tilde{v}^T=(v_1, v_2, \ldots,v_n)$.
We first assume that $C$ has no zero row. Then $r_i>r_{n+1}=0$ for $1\leq i\leq n$. By Lemma~\ref{l5.3}(ii) with
$(C', n)=(M, n+1)$,  we have $v'_j>v'_{n+1}$. Then
$c_{ij}=m_{ij}=c'_{ij}$ for the indices $1\leq i\leq n$ with $v_i\not=0$ and any $1\leq j\leq n$ by Theorem \ref{t3.1}(b).
Hence
\begin{equation}\label{ee6.2}\rho(C)\tilde{v}^T=\tilde{v}^TC=\tilde{v}^TC'(n+1|n+1)=\tilde{v}^T(fJ+(d-f)I).\end{equation}
 Since $\tilde{v}^T$ is a nonnegative left eigenvalue of the irreducible nonnegative matrix $fJ+(d-f)I$ for $\rho(C)$, we have  $\tilde{v}>0$. This together with $fJ+(d-f)I\geq C$ and (\ref{ee6.2}) will imply
$C=fJ+(d-f)I$, finishing the proof for the case under the assumption that $C$ has no zero row.
Assume that $C$ has $n-k$ zero rows for some $1\leq k\leq n-1$. Then there is a permutation matrix $P$ such that all zero rows of $PCP^T$ appear in the end, so the $(n-k)\times n$ submatrix  $PCP^T([k]|-]$ of $PCP^T$ is $0$ and the $k\times n$ submatrix $PCP^T[[k]|-]$ of $PCP^T$  has no zero row. Let $C_1=PCP^T[[k]|[k]]$
and $m'$ be the sum of entries in $C_1$. Notice that $\rho(C_1)=\rho(C)$ and  $m'\leq m$.
Applying the first part of the theorem to $C_1$, we have
$$\rho(C_1)\leq \frac{d-f+\sqrt{(d-f)^2+4m'f}}{2}\leq \frac{d-f+\sqrt{(d-f)^2+4mf}}{2}=\rho(C)=\rho(C_1),$$
forcing $m'=m$, $C_1$ has no zero row and $C_1=fJ_k+(d-f)I_k$.  Hence $PCP^T[[k]|[k])=0$ and this implies  $PCP^T=(fJ_k+(d-f)I_k)\oplus O_{n-k}$ and  $m=k(k-1)f+kd$.
\end{proof}


\section{$C'$ admitting an equitable quotient}\label{s3.5}

From now on the square matrix $C$ is nonnegative, and the eigenvalue $\rho(C)$ of $C$ is corresponding to a nonnegative left eigenvector $v^T$ by Theorem~\ref{PF}(i). Hence the assumption (iii) in Theorem~\ref{t3.1} and Theorem~\ref{t3.1'} immediately holds. In Lemma~\ref{l4.1} and Lemma~\ref{l5.3}, we know that a rooted matrix $C'$ and its translates are possessed of a rooted eigenvector for $\rho_r(C')$. In this section, we shall apply properties of the equitable quotient to provide matrices which are not translated from a rooted matrix but still have positive rooted eigenvectors.
We use this method to reduce the size of $C'$ in finding the bound $\lambda'$ of $\lambda=\rho(C)$ obtained in Theorem~\ref{t3.1} and Theorem~\ref{t3.1'}.

\begin{thm}\label{t6.1}
Let $C=(c_{ij})$ be a nonnegative $n\times n$ matrix with row-sum vector $(r_1,\ldots,r_n)^T$,  and
 $\Pi=\{\pi_1, \pi_2, \ldots, \pi_\ell\}$ a partition of $\{1, 2, \ldots, n\}$ with $n\in \pi_\ell$.
Let $C'=(c'_{ij})$ be an $n\times n$ matrix that admits
  an $\ell\times \ell$ equitable quotient matrix $\Pi(C')=(\pi'_{ab})$ of $C'$  with respect to $\Pi$ satisfying the following
(i)-(ii):
\begin{enumerate}
\item[(i)] $C[-|n)\leq C'[-|n)$ and
$\Pi(C')$ has row-sum vector  $\Pi(r')=$$(\pi(r')_1,$ $\pi(r')_2,$ $\ldots,$ $\pi(r')_\ell)^T$ with $\pi(r')_a=\max_{i\in \pi_a}r_i$ for $1\leq a\leq \ell$.
\item[(ii)] $\Pi(C')$ has a positive rooted eigenvector $\Pi(v')=$$(\pi(v')_1,$ $\pi(v')_2,$ $\ldots,$ $\pi(v')_\ell)^T$ for some nonnegative eigenvalue $\lambda'$.
\end{enumerate}
Then
\begin{equation}\label{e8.1}
\rho(C)\leq \lambda'.
\end{equation} Moreover, if $C$ is irreducible, then
$\rho(C)= \lambda'$ if and only if
\begin{enumerate}
\item[(a)] $r_i=\pi(r')_a$ \qquad for $1\leq a\leq \ell$ and $i\in \pi_a$, and
\item[(b)]  $c'_{ij}=c_{ij}$ \qquad for all $1\leq i, j\leq n$ such that for $1\leq b\leq \ell$ with $j\in \pi_b$ we have  $\pi(v')_b>\pi(v')_\ell$.
\end{enumerate}
\end{thm}

\begin{proof} Let $S$ be the $n\times \ell$ characteristic matrix of $\Pi.$
 From the construction of $\Pi$ and $C'$, $r'=S\Pi(r')=(r'_1,\ldots,r'_n)^T$ is the row-sum vector of $C'$, and
 $v'=S\Pi(v')$ is a positive rooted eigenvector of $C'$ for $\lambda'$  by Lemma~\ref{q1}. Since $C$ is nonnegative, there exists  a nonnegative left eigenvector $v^T$ of $C$ for $\rho(C)$ by Theorem~\ref{PF}(i).
  Hence $v^Tv'>0$.
   Thus assumptions (i)-(iv) of Theorem~\ref{t3.1} hold, concluding $\rho(C)\leq \lambda'$.

   Suppose that $C$ is irreducible. Then the above $v$ is positive. Hence the equivalent condition (b) of $\rho(C)=\lambda'$ in Theorem~\ref{t3.1} becomes
  $c'_{ij}=c_{ij}$ for $1\leq i\leq n$,$1\leq j\leq n-1$ with $v'_j>v'_n$, and this is equivalent to the (b) here from the structure of  $v'=S\Pi(v')$.
  The equivalent condition (a) here is immediate from that in Theorem~\ref{t3.1} since $r'_i=\pi(r')_a$ for $i\in\pi_a$.
\end{proof}

Notice that the irreducible assumption of $C$ in the second part of Theorem~\ref{tt6.1} is not necessary.
The following example shows that this is a must in that of Theorem~\ref{t6.1}.

 \begin{exam}
 Consider the following two $3\times 3$ matrices
 $$C=\begin{pmatrix}
 0 & 3 & 0 \\
 1 &  1 & 0\\
 1 & 0 & 0
 \end{pmatrix},\quad   C'=\begin{pmatrix}
 0 & 3 & 0 \\
 1 &  1 & 0\\
 1 & 0 & 1
 \end{pmatrix}\quad  {\rm and~~ }\Pi(C')=\begin{pmatrix} 0 & 3\\ 1 & 1 \end{pmatrix}$$
 is the equitable quotient matrix of $C'$  with respect to the partition $\Pi=\{\{1\}, \{2, 3\}\}.$
  Note that $C[-|3)= C'[-|3)$, and the row-sum vector $(3, 2, 1)^T$ of $C$ is majorized by the row-sum vector $(3, 2, 2)^T$ of $C'$.
 Since $\Pi(C')+I$ is positive and rooted, $\Pi(C')$ has a positive rooted eigenvector for $\lambda'=\rho(\Pi(C'))=(1+\sqrt{13})/2.$
 Hence assumptions (i)-(ii) in Theorem~\ref{t6.1} hold. By direct computing, $\rho(C)=(1+\sqrt{13})/2$, so the the equality in
(\ref{e8.1}) holds.  However, $r_2=2\not=1=r_3$, a contradiction to (a) in Theorem \ref{t6.1}.
 This contradiction is because of the reducibility of $C$.
 \end{exam}

The following is a dual version of Theorem~\ref{t6.1}.

\begin{thm}\label{t6.2}
Let $C=(c_{ij})$ be a nonnegative $n\times n$ matrix with row-sum vector $(r_1,\ldots,r_n)^T$,  and
 $\Pi=\{\pi_1, \pi_2, \ldots, \pi_\ell\}$ a partition of $\{1, 2, \ldots, n\}$ with $n\in \pi_\ell$.
Let $C'$ be an $n\times n$ matrix that admits
  an $\ell\times \ell$ equitable quotient matrix $\Pi(C')=(\pi'_{ab})$ of $C'$  with respect to $\Pi$ satisfying the following
(i)-(ii):
\begin{enumerate}
\item[(i)] $C[-|n)\geq C'[-|n)$ and
$\Pi(C')$ has row-sum vector  $\Pi(r')=$$(\pi(r')_1,$ $\pi(r')_2,$ $\ldots,$ $\pi(r')_\ell)^T$ with $\pi(r')_a=\min_{i\in \pi_a}r_i$ for $1\leq a\leq \ell$.
\item[(ii)] $\Pi(C')$ has a positive rooted eigenvector $\Pi(v')=$$(\pi(v')_1,$ $\pi(v')_2,$ $\ldots,$ $\pi(v')_\ell)^T$ for some nonnegative eigenvalue $\lambda'$.
\end{enumerate}
Then
\begin{equation}
\rho(C)\geq \lambda'.
\end{equation} Moreover, if $C$ is irreducible then
$\rho(C)= \lambda'$ if and only if
\begin{enumerate}
\item[(a)] $r_i=\pi(r')_a$ \qquad for $1\leq a\leq \ell$ and $i\in \pi_a$, and
\item[(b)] $c'_{ij}=c_{ij}$ \qquad for all $1\leq i, j\leq n$ such that for $1\leq b\leq \ell$ with $j\in \pi_b$ we have  $\pi(v')_b>\pi(v')_\ell$.
\end{enumerate}\qed
\end{thm}

\begin{rem}
\begin{enumerate}
\item[(i)] The positive assumption of $\Pi(v')$ in (ii)  of  Theorem~\ref{t6.2} can be removed in concluding the first part $\rho(C)\geq \lambda'$.
The following is a proof:

{\it Proof.} From (i) and referring to (\ref{C'Q}), we have $CQ\geq C'Q\geq 0$.
Let $v'=S\Pi(v')$ be a rooted eigenvector of $C'$ for $\lambda'$ as shown in the above proof.
Then $u=Q^{-1}v'$ is nonnegative by Lemma~\ref{l3.15}(i), so
$Cv'=CQu\geq C'Qu= C'v'=\lambda'v'.$
Since $v'$ is nonnegative,  $\rho(C)\geq \lambda'$ by Theorem~\ref{PF}(iii). \qed

\item[(ii)] The following counterexample shows that to conclude $\rho(C)\leq \lambda'$, the positive assumption of $\Pi(v')$ in (ii)  of  Theorem~\ref{t6.1} can not  be removed:
    $$C=C'=\begin{pmatrix}1& 2\\ 0 & 2\end{pmatrix}, \lambda'=1, v'=\begin{pmatrix} 1\\ 0\end{pmatrix}, \rho(C)=2, v^T=(0, 1),$$
    where the trivial partition $\Pi=\{\{1\}, \{2\}\}$ of $\{1, 2\}$ is adopted.
\end{enumerate}
\end{rem}

We provide an example in applying Theorem~\ref{t6.1}.

\begin{exam}\label{exammain}
Consider the following two $7\times 7$ matrices $C$ and $C'$ expressed below under the partition $\Pi=\{\{1, 2, 3\}, \{4, 5\}, \{6, 7\}\}:$
\begin{equation}\label{e8.6}C=\left(
\begin{tabular}{ccc|cc|cc} 2 & 1 & 3 & 3 & 3 & 12 & 0\\ 4 & 2 & 1 & 4 & 2 & 6 & 4\\ 2 & 3 & 1 & 4 & 1 & 8 &3 \\
\hline
3 & 5 & 3 & 1 & 1& 3 & 4\\ 5 & 6 & 1 & 1 & 0 & 3 &3\\
\hline
0 & 2 & 1 & 2 & 2 & 6 & 0\\  2 & 2 & 0 & 2 &1& 1 &4\end{tabular}\right), \qquad
  C'=\left(\begin{tabular}{ccc|cc|cc} 2 & 2 & 3 & 3 & 3 & 12 & -1\\ 4 & 2 & 1 & 4 & 2 & 6 & 5\\ 2 & 3 & 2 & 4 & 2 & 8 & 3 \\ \hline
 4& 5 & 3 & 1 & 1& 3 & 3\\
 5 & 6 & 1 & 1 & 1 & 3 &3\\
 \hline 1 & 2 & 1 & 2 & 2 & 6 & -1\\  2 & 2 & 0 & 2 &2& 1 &4\end{tabular}\right).\end{equation}
Apparently, $C[-|7)\leq C'[-|7),$ and
 the row-sum vector
$(24, 23, 22, 20,$ $19, 13, 12)^T$ of $C$ is majorized by the  row-sum vector $(24, 24, 24, 20, 20, 13, 13)^T$ of $C'$,
so assumption (i)  of Theorem~\ref{t6.1} holds.
Note that $C'$ is not rooted  and neither of its translates.
Since  $C'$ has equitable quotient matrix
$$\Pi(C')=\begin{pmatrix} 7 & 6 & 11 \\ 12 & 2 & 6\\ 4& 4 & 5\end{pmatrix},$$
in which $\Pi(C')+2I$ is rooted, so assumption (ii) of Theorem~\ref{t6.1} holds for $\lambda'=\rho_r(C')$ by Lemma~\ref{l4.1} and Lemma~\ref{l5.3}(i).
Hence by Theorem~\ref{t6.1}, $\rho(C)\leq \rho_r(\Pi(C'))\approx 18.6936.$

If applying Lemma~\ref{lem1.2} by constructing the following nonnegative matrix $C''$ that majors $C$, and find its equitable quotient matrix $\Pi(C'')$ with respect to the above partition $\Pi$:
$$C''=\left(
\begin{tabular}{ccc|cc|cc} 2 & 2 & 3 & 3 & 3 & 12 & 0\\ 4 & 2 & 1 & 4 & 2 & 6 & 6\\ 2 & 3 & 2 & 4 & 2 & 8 &4 \\
\hline
4 & 5 & 3 & 1 & 1& 3 & 4\\ 5 & 6 & 1 & 1 & 1 & 3 &4\\
\hline
 1& 2 & 1 & 2 & 2 & 6 & 0\\  2 & 2 & 0 & 2 &2& 2 &4\end{tabular}\right), \qquad \Pi(C'')=\begin{pmatrix} 7 & 6 & 12 \\ 12 & 2& 7\\ 4& 4 & 6\end{pmatrix},$$
 one will find the upper bound
$$\rho(C'')=\rho (\Pi(C'')\approx 19.4$$
of $\rho(C)$ which is larger than the previous one.
\end{exam}

\section{More irrelevant columns}

Considering the part $\pi_\ell$ of column indices of $C$ and $C'$ in the assumption (i) of Theorem~\ref{t6.1},
the assumption $C[-|\pi_\ell]\leq C'[-|\pi_\ell]$ for $C'$ is not really necessary. We might replace the columns indexed by
$\pi_\ell$ in $C'$ by any other columns and adjust the values in the last column of keeping the row-sums of $C'$ unchanged.
In this situation, the columns of $C'$ indexed by $\pi_\ell$ are irrelevant columns (in the comparison of $C$ and $C'$).
For example in Example~\ref{exammain},
the values in the  $6$-th column of $C'$ can be changed to any values (e.g., $(a, b, c, d, e, f, g)^T$), if the values in the $7$-th column of $C'$ make the corresponding change (e.g., $(11-a, 11-b, 11-c, 6-d, 6-e, 5-f, 5-g)^T$ correspondingly), i.e., columns $6$ and $7$ of $C'$ are irrelevant.  The following theorem generalizes this idea when restricting $\Pi[C']$ in Theorem~\ref{t6.1} to be a rooted matrix or its translate.

\begin{thm}\label{main3}
Let  $\Pi=\{\pi_1, \pi_2, \ldots, \pi_\ell\}$ be a partition of $[n]$ with $n\in \pi_\ell$, and
$C$ be an $n\times n$ nonnegative matrix with row-sums $r_1\geq r_2\geq \cdots \geq r_n$.
For $1\leq a \leq \ell$ and $1\leq b\leq \ell-1,$
choose  $r''_a,$ $c''_{ab}$ such that
$$
\left\{
\begin{array}{rcl}
r''_a&=&\max_{i\in \pi_a} r_i;\\
c''_{ab}   &\geq&\sum_{j\in \pi_b} c_{ij}\qquad \hbox{for all } i\in \pi_a;\\
c''_{ab}   &\geq& c''_{\ell b}> 0 \qquad \hbox{for } a\not=b
\end{array}
\right.$$
and let
$$c''_{a\ell}=r''_a-\sum_{j=1}^{\ell-1} c''_{aj}.$$
Then the $\ell\times\ell$ matrix $C''=(c''_{ab})$  has a positive rooted eigenvector $v''=(v''_1, v''_2, \ldots, v''_\ell)^T$ for $\rho_r(C'')$,
and  $\rho(C)\leq \rho_r(C'').$
Moreover, if $C$ is irreducible, then $\rho(C)= \rho_r(C'')$ if and only if
\begin{enumerate}
\item[(a)] $r_i=r''_a$ \qquad for $1\leq a\leq \ell$ and $i\in \pi_a$, and
\item[(b)]
$\sum_{j\in \pi_b}c_{ij} =c''_{ab}\qquad \hbox{for all $1\leq a, b\leq \ell$ with $v''_b>v''_\ell$ and  $i\in \pi_a$}.$
\end{enumerate}
\end{thm}

\begin{proof} From the construction of $C''$,
$C''+dI$ is a positive rooted matrix for some $d$ large enough, so $C''$ has a positive rooted eigenvector for $\rho_r(C'')$ by Lemma~\ref{l4.1} and Lemma~\ref{l5.3}(i).
In view of the construction of $C'$ in Example~\ref{exammain}, we construct an $n\times n$ matrix $C'$ such that $C'$ has the equitable quotient matrix $\Pi(C')=C''$ and assumptions (i)-(ii) of Theorem \ref{t6.1} hold for $\lambda'=\rho_r(\Pi(C'))$.
Hence the remaining follows from the conclusion of Theorem \ref{t6.1}.
\end{proof}

\begin{rem}
Theorem~\ref{dz2013} is a special case  of Theorem~\ref{main3} with $\Pi=\{\{1\}, \{2\}, \ldots, \{\ell-1\}, \{\ell, \ell+1, \ldots, n\}\}$ and $C''=M_{\ell}(d,f,r_1,r_2,\ldots,r_\ell)$ as shown in (\ref{ne1.2}). The equality case needs to apply Lemma~\ref{l7.1}(iii) by choosing a new $\ell$ to be the least $t$ such that $r_t=r_\ell$.
By using the more irrelevant columns idea in Theorem~\ref{main3}, the assumption  $f:=\max_{1\leq i\not=j\leq n} c_{ij}$ and $d:=\max_{1\leq i\leq n}c_{ii}$ in Theorem~\ref{dz2013} can be replaced by the possible smaller number $f:=\max_{1\leq i\leq n,1\leq j\leq \ell-1,i\ne j} c_{ij}$ and $d:=\max_{1\leq i\leq \ell-1} c_{ii}$ respectively.
\end{rem}

The following is a dual theorem of Theorem~\ref{main3}.
\begin{thm}\label{main4}
Let  $\Pi=\{\pi_1, \pi_2, \ldots, \pi_\ell\}$ be a partition of $[n]$ with $n\in \pi_\ell$, and
$C$ an $n\times n$ nonnegative matrix with row-sums $r_1\geq r_2\geq \cdots \geq r_n$.
For $1\leq a \leq \ell$ and $1\leq b\leq \ell-1,$
choose  $r''_a,$ $c''_{ab}$ such that
\begin{equation}\label{e9.1}
\left\{
\begin{array}{rcl}
r''_a&=&\min_{i\in \pi_a} r_i;\\
c''_{ab}   &\leq&\sum_{j\in \pi_b} c_{ij}\qquad \hbox{for all } i\in \pi_a;\\
c''_{ab}   &\geq& c''_{\ell b}> 0 \qquad \hbox{for } a\not=b
\end{array}
\right.\end{equation}
and let
\begin{equation}\label{e9.2}
c''_{a\ell}=r''_a-\sum_{j=1}^{\ell-1} c''_{aj}.
\end{equation}
Then the $\ell\times\ell$ matrix $C''=(c''_{ab})$  has a positive rooted eigenvector $v''=(v''_1, v''_2, \ldots, v''_\ell)^T$ for $\rho(C'')$,
and
$\rho(C)\geq \rho(C'').$
Moreover, if $C$ is irreducible, then $\rho(C)= \rho(C'')$ if and only if
\begin{enumerate}
\item[(a)] $r_i=r''_a$ \qquad for $1\leq a\leq \ell$ and $i\in \pi_a$, and
\item[(b)]
$\sum_{j\in \pi_b}c_{ij} =c''_{ab}\qquad \hbox{for all $1\leq a, b\leq \ell$ with $v''_b>v''_\ell$ and  $i\in \pi_a$}.$
\end{enumerate}
\end{thm}

\section{Some new lower bounds of spectral radius}

We shall apply Theorem~\ref{main4} to obtain a lower bound of $\rho(C)$ for a nonnegative matrix $C$.

\begin{thm}\label{main5}
Let
$C=(c_{ij})$ be an $n\times n$ nonnegative matrix with row-sums $r_1\geq r_2\geq \cdots \geq r_n$.
For $1\leq t< n,$ let $\Pi_t=\{\{1, \ldots, t\}, \{t+1, \ldots, n\}\}$ be a partition of $[n]$.
Let $d=\max_{t<i\leq n}c_{ii}$ and $f=\max_{1\leq i\leq n, t<j\leq n, i\not=j}c_{ij}.$
Assume that $0<r_n-(n-t-1)f-d\leq r_t-(n-t)f.$
Then
\begin{equation}\label{ne9.1}\rho(C)\geq \frac{r_t-f+d+\sqrt{(r_t-(2n-2t-1)f-d)^2+4(n-t)f(r_n-(n-t-1)f-d)}}{2}.\end{equation}
Moreover, if $C$ is irreducible and $f>0$, then the equality holds in (\ref{ne9.1}) if and only if $r_1=r_n$ or
\begin{enumerate}
\item[(a)] $r_1=r_t$ and $r_{t+1}=r_n$, and
\item[(b)]
$\sum_{j\in [t]}c_{ij} =r_t-(n-t)f\quad \hbox{for all  $i\in [t]$},$ and

$\sum_{j\in [t]}c_{ij} =r_n-(n-t-1)f-d\quad \hbox{for all  $t<i\leq n$}$.
\end{enumerate}
\end{thm}
\begin{proof} The lower bound of $\rho(C)$ in (\ref{ne9.1}) follows by
applying Theorem~\ref{main4} with the following positive rooted matrix
$$C''=\begin{pmatrix}
 r_t-(n-t)f & (n-t)f \\
 r_n-(n-t-1)f-d& (n-t-1)f+d
 \end{pmatrix},$$
 which has row-sum vector $(r_t, r_n)^T$ and the assumptions in (\ref{e9.1}) and (\ref{e9.2}) of Theorem~\ref{main4} hold from the assumptions.
 Note that $C''$ has a positive rooted eigenvector  $(v''_1, v''_2)^T$
 for $\rho(C'')$ by  Lemma~\ref{l5.3}(i), and the value
  $\rho(C'')$ is as shown  in  the right of (\ref{ne9.1}).
To study the equality case in (\ref{ne9.1}), we apply conditions (a)-(b) in Theorem~\ref{main4},
in which (a) is exactly the (a) of this theorem.
If $v''_1>v''_2$ then the condition (b) of this theorem is exactly the (b) of Theorem~\ref{main4}.
Notice that  $v''_2=v''_1$ if and only if  $\rho(C'')=r_t=r_n$ by Theorem \ref{PF} 
using the irreducible property of  $C''$. This is also equivalent to  $r_1=r_n$ under the condition (a).
\end{proof}

The following corollary restricts Theorem~\ref{main5} to binary matrix $C$.
\begin{cor}\label{cor10.2}
Let
$C=(c_{ij})$ be an $n\times n$ (0, 1) matrix with row-sums $r_1\geq r_2\geq \cdots \geq r_n>0$,
and choose $t\geq n-r_n+1$ and $t\leq n$. Then
\begin{equation}\label{ne9.2}
\rho(C)\geq \frac{r_t+\sqrt{r_t^2-4(r_n-1)(r_t-r_n)}}{2}.\end{equation}
Moreover, if $C$ is irreducible, then equality holds in (\ref{ne9.2}) if and only if $r_1=r_n$ or
\begin{enumerate}
\item[(a)] $r_1=r_t$ and $r_{t+1}=r_n$, and
\item[(b)]
$\sum_{j\in [t]}c_{ij} =r_t-(n-t)\quad \hbox{for all  $i\in [t]$},$ and

$\sum_{j\in [t]}c_{ij} =r_n-(n-t)\quad \hbox{for all  $t<i\leq n$}$.
\end{enumerate}\qed
\end{cor}
\begin{proof}
If $t=n$ then (\ref{ne9.2}) becomes $\rho(C)\geq r_n$, so the corollary follows   from Lemma~\ref{lem1.1}.
Assume $t<n$. Since the assumptions in Theorem~\ref{main5} clearly holds with $d=f=1,$ the corollary also holds by \eqref{ne9.1} in this case.
\end{proof}

One can easily check that the right of (\ref{ne9.2}) is at least $r_n$ (with equality iff $r_t=r_n$), so the above lower bound is better than the known one $r_n$
in Lemma~\ref{lem1.1}.

\end{document}